\documentclass[reqno]{amsart}
\usepackage{amssymb,latexsym}
\usepackage{amsfonts,mathrsfs}
\usepackage[margin=1.4in]{geometry}
\usepackage[all]{xy}
\numberwithin{equation}{section}

\newtheorem{theorem}{Theorem}[section]
\newtheorem{lemma}[theorem]{Lemma}

\newtheorem{corollary}[theorem]{Corollary}
\theoremstyle{definition}

\theoremstyle{definition}

\def \i {\sqrt{-1}}
\def \pbp {\partial\bar\partial}

\def\XXint#1#2#3{{\setbox0=\hbox{$#1{#2#3}{\int}$ }
\vcenter{\hbox{$#2#3$ }}\kern-.58\wd0}}

\begin{document}
\title{A note on the generalization of the Kodaira embedding theorem}
\author{Chao Li,  Xi Zhang and QiZhi Zhao}
\address{Chao Li\\School of Mathematical Sciences\\
University of Science and Technology of China\\
Hefei, 230026,P.R. China\\}\email{leecryst@mail.ustc.edu.cn}

\address{Xi Zhang\\School of Mathematical Sciences\\
University of Science and Technology of China\\
Hefei, 230026,P.R. China\\ } \email{mathzx@ustc.edu.cn}
\address{Qizhi Zhao\\School of Mathematical Sciences\\
University of Science and Technology of China\\
Hefei, 230026,P.R. China\\}\email{zqz97@mail.ustc.edu.cn}
\thanks{The authors were supported in part by NSF in China,  No.11625106, 11571332 and 11721101.}
\maketitle
\begin{abstract}
In this paper, by using analytical methods  we obtain a generalization of the famous Kodaira embedding theorem.
\end{abstract}

\vspace{15pt}

\section{Introduction}

\vspace{10pt}

The famous Kodaira embedding theorem (see e.g. \cite[\S 1.4]{G69}) asserts  that any compact complex manifold admitting a positive line bundle can be embedded to some projective space $\mathbb CP^N$. In this paper, we consider a generalization associated with holomorphic fibration.

\begin{theorem}\label{mthm} Let $f:X\rightarrow Y$ be a non-singular holomorphic fibration with $Y$ and all the fibers being compact and connected. If there exists a line bundle $L$ on $X$ such that for any $y \in Y$, $L|_{f^{\!-\!1}(y)}\rightarrow f^{\!-\!1}(y)$ is positive. Then we can find a holomorphic vector bundle $\pi_E:E\rightarrow Y$ and a holomorphic embedding $j:X\rightarrow P(E^*)$ satisfying $f=\pi_{P(E^*)}\circ j$, where $\pi_{P(E^*)}$ is the natural projection from $P(E^*)$ to $Y$.
\end{theorem}

In algebraic geometry, $L$ is called $f$-ample or ample relative to $f$ (see e.g. \cite[\S 1.7]{L04I}), and one may give an analogous proof of Theorem \ref{mthm}. However, in this paper, we only give a proof using analytical method.

\medskip

We consider the following vector bundle in form
\begin{equation}E_k=\mathop{\bigsqcup}\limits_{p\in Y} H^0(L^k|_{f^{\!-\!1}(p)},f^{\!-\!1}(p)).
\end{equation}
where $k\geq 0$. By choosing sufficiently large $k$, we can find a proper holomorphic structure on $E_k$, and construct fiber-wise holomorphic embeddings following the proof of the Kodaria embedding theorem. These fiber-wise embedding maps together form the required global one.

\vspace{15pt}

\section{Looking back on the Kodaira embedding theorem}

\vspace{10pt}

Let $M$ be a compact complex manifold and $L$ be a positive line bundle on $M$. We look back to the construction of the embedding map from $M$ to certain $\mathbb CP^N$ in the proof of the Kodaira embedding theorem.

\medskip

For any $k\geq 0$, let
\begin{equation}
N_k=\dim H^0(L^k,M)-1.
\end{equation}
Since $L$ is positive, when $k$ is sufficiently large, $L^k$ is generated by its global holomorphic sections, i.e. for any $p\in M$ and $v\in L^k|_p$, there exist a global holomorphic section $s$ of $L^k$, such that $s(p)=v$. Let $\{s_{0},s_{1},\cdots,s_{N_k}\}$ be a basis of $H^0(L^k,M)$, then for any $p\in M$, some $s_i(p)$ is nonzero. By this fact, we can define a holomorphic map $\varphi_{k}:M\rightarrow CP^{N_k}$ by
\begin{equation}\varphi_k(p)=[f_0(p),f_1(p),\cdots,f_{N_k}(p)],
\end{equation}
where $f_i e=s_i$ and $e$ is an arbitrary nonzero local holomorphic section of $L^k$ around $p$. $\varphi_k$ is well-defined, and it becomes an embedding when $k$ is sufficiently large.

\medskip

In fact this construction is ``canonical" in some sense. Let $V_k^*$ be the dual space of $H^0(L^k,M)$. For any $p\in M$ and any $e^*\in L^k|_p\setminus\{0\}$, $e^*$ determines a linear homogeneous function $s^*_{e^*}$ on $H^0(L^k,M)$, i.e. an element in $V_k^*$, by setting
\begin{equation}s^*_{e^*}(s)=e^*(s(p)).
\end{equation}
If there exists an $s\in H^0(L^k,M)$ such that $s(p)\neq 0$, then $s^*_{e^*}$ is nozero and determines an element $[e^*]=[s^*_{e^*}]$ in $P(V_k^*)\cong \mathbb CP^{N_k}$, which doesn't depend on the choice of $e^*\in L^{\!-\!k}|_p\setminus\{0\}$. When $L^k$ is generated by its global holomorphic sections, the map $\phi_k:M\rightarrow P(V_k^*)$ defined by $p\mapsto [e^*]$ is well defined and holomorphic.

Using the basis $\{s_{0},s_{1},\cdots,s_{N_k}\}$ of $H^0(L^k,M)$, we can define an isomorphism $\rho_k:V_k\rightarrow \mathbb CP^n$
\begin{equation}\rho_k([s^*])=[s^*(e_0),s^*(e_1),\cdots,s^*(e_{N_k})].
\end{equation}
One can check that $\rho\circ\phi_k=\varphi_k$, so $\phi_k$ and $\varphi_k$ are essentially the same one.

\medskip

Let $f:X\rightarrow Y$ and $L$ be as in Theorem \ref{mthm}, and let $E_k$ as mentioned  in the introduction. By the smoothness of the fibration and the compactness of $Y$, it is not difficult to prove that $\dim E_k|_p$ is a constant and $E_k$ admits a ``natural" differential structure such that as a differential sheaf
\begin{equation}
E_k\cong f_*(\mathcal O_{X/Y}(L^k)),
\end{equation}
where $\mathcal O_{X/Y}(L^k)$ is the sheaf generated by the smooth sections of $L^k$ which are fiber-wisely holomorphic. Moreover, for any $p\in Y$, when $k$ is sufficiently large, $L^k|_{f^{\!-\!1}(p)}$ is generated by its global holomorphic sections and the canonical holomorphic map $\phi_{k,p}:f^{\!-\!1}(p)\rightarrow P((E_k|_p)^*)$ is an embedding, equivalently $L^k|_{f^{\!-\!1}(p)}$ is very ample. Later we will show that we can find a uniform $k$ such that all $L^k|_{f^{\!-\!1}(p)}$ is very ample.

\vspace{15pt}

\section{Extending global holomorphic sections on a fiber}

\vspace{10pt}

Let $f:X\rightarrow Y$ and $L$ be as in Theorem \ref{mthm}. As mentioned in the introduction, we set
\begin{equation}E_k=\mathop{\bigsqcup}\limits_{p\in Y} H^0(L^k|_{f^{\!-\!1}(p)},f^{\!-\!1}(p)).
\end{equation}
Defining a holomorphhic structure on some $E_k$ is equivalent to defining proper holomorphic sections of $E_k$. A natural way is to consider the sections which can be induced by holomorphic sections of $L^k$. For any non-empty open set $U\subset Y$, we denote $F(U,E_k)$ to be the set of sections $s$ of $E_k$ on $U$ satisfying
\begin{equation}s(p)=\tilde s|_{f^{\!-\!1}(p)},\qquad \forall p\in U,
\end{equation}
for some $\tilde s\in \Gamma(f^{\!-\!1}(U),L^k)$. In order that $\{F(U,E_k)\}$ truly determines a holomorphic structure on $E_k$, we only need to make sure  that for any $p\in Y$ and $s\in H^0(f^{\!-\!1}(p),L^k|_{f^{\!-\!1}(p)})$, $s$ can be extended to a local holomorphic section of $L^k$ on a neighborhood of $f^{\!-\!1}(p)$.  In fact, the assignment $U\mapsto F(U,E_k)$ corresponds to the coherent analytic sheaf $f_*(L^k)$. When the mentioned condition holds, we actually have $E_k\cong f_*(L^k)$.

\medskip

The following lemma is due to Berndtsson.
\begin{lemma}[{\cite[Theorem 8.1]{B09}}] \label{ber09}
Let $\pi:M\rightarrow U$ be a non-singular holomorphic fibration over the unit ball $U\subset \mathbb C^n$, with compact and connected fibers. Assume $M$ is K\"ahler and $L$ is a semi-positive line bundle over $M$. Then any $s\in H^0((L\otimes \mathcal K_{M/U})|_{\pi^{-1}(0)},\pi^{-1}(0))$ can be extended to a local holomorphic section of $L\otimes \mathcal K_{M/U }$ on $M$.
Consequently $\pi_*(L\otimes \mathcal K_{M/U })$ is a holomorphic vector bundle.
\end{lemma}

For our usage, we need the following lemma
\begin{lemma}\label{lf}
Let $\pi:M\rightarrow N$ be a non-singular holomorphic fibration  with compact and connected fibers. And let $L$ be a holomorphic line bundle over $M$ whose restriction to any fiber $\pi^{\!-\!1}(p)\ (p\in N)$ is positive. If $U$ is a domain on $N$ satisfying that there exists a local holomorphic chart $\{\tilde U;z^1,\cdots z^n\}$ such that $\overline U=\{|z|\leq1\}$. Then there exists a $k_0\geq 1$, such that when $k\geq k_0$, for any $q\in U$, any $s\in H^0(L^k|_{\pi^{\!-\!1}(q)},\pi^{\!-\!1}(q))$ can be extended to a local holomorphic section of $L^k$ on a neighborhood of $\pi^{-1}(q)$.
Consequently $\pi_*(L^k)$ is a holomorphic vector bundle over $U$.
\end{lemma}

\begin{proof}
First we show that over $\pi^{\!-\!1}(\overline U)$, $L$ admits a Hermitian metric with positive curvature. For any $p\in N$, $L|_{\pi^{\!-\!1}(p)}$ is positive, so we can find a smooth $h_p$ on $L$ over a neighborhood of $\pi^{\!-\!1}(p)$, whose restriction to $\pi^{\!-\!1}(p)$ admits positive curvature. By the smoothness of $h_p$, we can find an open neighborhood $U_p$ of $p$ such that for any $q\in U_p$ the restriction of $h_p$ over $\pi^{\!-\!1}(q)$ admits positive curvature. Since $\overline U$ is compact, we can find finite many $p_i$ such that $\overline U\subset \mathop{\bigcup}\limits_i U_{p_i}$. Let $\{\alpha_i\}$ be an partition of unity of  $\overline U$ subject to $\{U_{p_i}\}$. We can define a Hermitian metric $h$ on $L$ as
\begin{equation}h=\prod_i h_{p_i}^{\alpha_i\circ \pi}.
\end{equation}
One can easily check that  for any $q\in \overline U$ the restriction of $h$ over $\pi^{\!-\!1}(q)$ admits positive curvature.

By the local holomorphic chart $\{\tilde U;z^1,\cdots z^n\}$, the function $\phi=|z|^2$ satisfies $\i\pbp \phi>0$ on $\overline U$. Then it is easy to see, when $K\geq 0$ is sufficiently large, the Hermitian metric $he^{-K\phi\circ f}$ over $\pi^{\!-\!1}(\overline U)$ admits positive curvature.

\medskip

Since $L|_{\pi^{\!-\!1}(\overline U)}$ admits a Hermitian metric with positive curvature, $\pi^{\!-\!1}(\overline U)$ is K\"ahler. At the same time, when $k$ is sufficiently large, $(L^k\otimes \mathcal K_{M/N}^{\!-\!1})|_{\pi^{\!-\!1}(\overline U)}$ admits a Hermitian metric with positive curvature. By applying Theorem \ref{ber09}, we can complete the proof.
\end{proof}

Back to Theorem \ref{mthm}, by Lemma \ref{lf} and the compactness of $Y$, when $k$ is sufficiently large, $E_k$ admits a canonical holomorphic structure such that $E_k\cong \pi_*(L^k)$.

\vspace{15pt}

\section{Proof of Theorem \ref{mthm}}

\vspace{10pt}

First, we fix some sufficiently large $k_0$ such that for any $k\geq k_0$, $E_k$ admits a canonical holomorphic structure such that $E_k\cong \pi_*(L^k)$.

For any $p\in Y$, since $L|_{f^{\!-\!1}(p)}$ is positive, by the (proof of) the Kodaia embedding theorem, when $k$ is sufficiently large, the canonical map $\phi_{k,p}:f^{\!-\!1}(p)\rightarrow P((E_k|_p)^*)$ is a well-defined embedding, equivalently $L^k|_{f^{\!-\!1}(p)}$ is very ample. To complete the proof, we need to find a uniform $k\geq k_0$ such that for any $p\in Y$, $L^k|_{f^{\!-\!1}(p)}$ is very ample.

For any $p\in Y$, let $k_p\geq k_0$ be an integer such that for any $k\geq k_p$, $L^k|_{f^{\!-\!1}(p)}$ is very ample. Fix a $k\geq k_p$, by the fact $E_k\cong f_*(L^k)$, we can find an open neighborhood $U_{k,p}$ of $p$, and  local holomorphic sections $e_0,e_1,\cdots,e_{N_k}$ such that for any $q\in U_{k,p}$, $\{e_0|_{f^{\!-\!1}(q)},e_1|_{f^{\!-\!1}(q)},\cdots, e_{N_k}|_{f^{\!-\!1}(q)}\}$ is a basis of $E_k|_q\cong H^0(f^{\!-\!1}(q),L^k|_{f^{\!-\!1}(q)})$.
By the smoothness of $e_0,e_1,\cdots,e_{N_k}$, for any $q$ in a certain smaller open neighborhood $U_{k,p}'$, $e_0|_{f^{\!-\!1}(q)},e_1|_{f^{\!-\!1}(q)},\cdots, e_{N_k}|_{f^{\!-\!1}(q)}$ have no common zeros so the canonical map $\phi_{k,q}$ is well-defined.
Further, all $\phi_{k,q}\ (q\in U_{k,p}')$ forms a holomorphic map $\phi_{k}:f^{\!-\!1}(U_{k,p}')\rightarrow P(E_k^*)$. We have the following trivialization map $\rho_{k,p}:P(E_k^*)|_{U_{k,p}'}\rightarrow U_{k,p}'\times \mathbb CP^{N_k}$
\begin{equation}\rho_{k,p}(s^*)=(\pi_k(s^*),[s^*(e_0),s^*(e_1),\cdots,s^*(e_{N_k})]),
\end{equation}
where $\pi_k:P(E_k^*)|_{U_{k,p}'}\rightarrow U_{k,p}'$ is the natural projection. Then we have the following expression for $\phi_{k}$
\begin{equation}\rho_{k,p}\circ\phi_{k}(x)=(\pi_k(l^*),[l^*(e_0),l^*(e_1),\cdots,l^*(e_{N_k})]),
\end{equation}
where $l^*$ is an arbitrary nonzero local holomorphic section of $L^*|_{f^{\!-\!1}(U_{k,p}')}$. By this expression it is easy to check that $\phi_k$ is holomorphic.
The fact that $\phi_{k,p}$ is an embedding implies that for any $q$ in a certain even smaller open neighborhood $U_{k,p}''$, $\phi_{k,q}$ is also an embedding. Consequently $\phi_k$ is an embedding on $f^{\!-\!1}(U_{k,p}'')$.

\medskip

Let $p$, $k_p$ and $U_{k,p}''$ be as above, we set $V_p=\bigcap_{k=k_p}^{2k_p\!-\!1}U_{k,p}''$. Then for any $k=k_p,k_p\!+\!1,\cdots,2k_p\!-\!1$, and $q\in V_p$, $L^k|_{f^{\!-\!1}(q)}$ is very ample. By the following Corollary \ref{vamp2}, for any $k\geq k_p$ and $q\in V_p$, $L^k|_{f^{\!-\!1}(q)}$ is very ample.

\begin{lemma}\label{vamp1}
Let $M$ be a compact complex manifold. If $L_1\rightarrow M$ is a very ample line bundle, and $L_2\rightarrow M$ is a holomorphic line bundle which is globally generated. Then $L_1\otimes L_2$ is very ample.
\end{lemma}

\begin{corollary}\label{vamp2}
Let $L\rightarrow M$ be an ample line bundle over a compact complex manifold. If there exits a $k_0\geq 0$, such that $L^k$ is very ample for $k=k_0,k_0\!+\!1,\cdots,2k_0\!-\!1$. Then $L^k$ is very ample for all $k\geq k_0$.
\end{corollary}

Another way to find $k_p$ and $V_p$ with similar properties is to extend the origin proof of Kodaira embedding theorem to a family setting. We omit the detailed argument here.

\medskip

By the above argument and the compactness of $Y$, we can easily find a uniform integer $k_1\geq k_0$, such that for all $k\geq k_1$ and $p\in X$, $L^k|_{f^{\!-\!1}(p)}$ is  very ample. Consequently the canonical map $\phi_k:X\rightarrow P(E_k^*)$ is a well-defined embedding. This concludes the proof.

\vspace{20pt}

\bibliographystyle{plain}

\end{document}